\documentclass[11pt]{amsart}
\usepackage[margin=1in]{geometry}
\usepackage[T1]{fontenc}
\usepackage[utf8]{inputenc}
\usepackage{lmodern}
\usepackage{amsmath,amssymb,amsfonts,amsthm,mathtools}
\usepackage{enumitem}
\usepackage{microtype}
\usepackage[colorlinks=true,linkcolor=blue,citecolor=blue,urlcolor=blue]{hyperref}

\setlist[itemize]{leftmargin=2em}
\setlist[enumerate]{leftmargin=2em}

\newtheorem{theorem}{Theorem}[section]
\newtheorem{proposition}[theorem]{Proposition}
\newtheorem{corollary}[theorem]{Corollary}
\newtheorem{lemma}[theorem]{Lemma}

\theoremstyle{definition}
\newtheorem{definition}[theorem]{Definition}
\newtheorem{example}[theorem]{Example}
\theoremstyle{remark}
\newtheorem{remark}[theorem]{Remark}

\DeclareMathOperator{\sepd}{sep}

\DeclareMathOperator{\Poisson}{Poisson}

\newcommand{\E}{\mathbb{E}}
\newcommand{\Pbb}{\mathbb{P}}
\newcommand{\R}{\mathbb{R}}

\newcommand{\fall}[2]{(#1)_{#2}}
\newcommand{\normTV}[1]{\left\lVert #1\right\rVert_{\mathrm{TV}}}

\title[Growing-block colored top-m-to-random]{Cutoff profiles for colored top-\texorpdfstring{$m$}{m}-to-random shuffles with growing block size}

\author[I. Z. Feng]{Ivan Z. Feng}
\address{Department of Mathematics, University of Southern California, Los Angeles, CA 90089, USA}
\email{ifeng@usc.edu}
\urladdr{https://dornsife.usc.edu/ivan/}

\date{June 28, 2026}


\begin{document}

\begin{abstract}
We study the \(p\)-colored top-\(m\)-to-random shuffle on \(C_p\wr S_n\) when the block size \(m=m_n\) grows with \(n\).  Let \(E_{k_n}^{(m_n)}\) be the number of labels never touched after \(k_n\) independent uniform \(m_n\)-subset draws, and set
\[
        b_n=n-m_n,\qquad q_n=\frac{b_n}{n},\qquad
        \lambda_n=nq_n^{k_n}.
\]
We prove that if \(\lambda_n\to\lambda\in(0,\infty)\) and \(b_n\to\infty\), then
\[
        E_{k_n}^{(m_n)}\Rightarrow\Poisson(\lambda).
\]
Combining this with the exact nested-set reduction for colored top-\(m\)-to-random shuffles, we obtain growing-block total variation, separation, and integrated likelihood-ratio profiles.  In particular, if \(Q_{n,p}^{(m_n)}\) is the one-step law and \(U_{n,p}\) is uniform on \(C_p\wr S_n\), then
\[
\sepd\!\left((Q_{n,p}^{(m_n)})^{*k_n},U_{n,p}\right)\to
\begin{cases}
1-e^{-\lambda}(1+\lambda),&p=1,\\
1-e^{-\lambda},&p\ge2.
\end{cases}
\]
The criterion applies to small blocks, proportional blocks, and near-full blocks.
\end{abstract}

\maketitle

\section{Introduction}

The top-to-random shuffle is a classical example in which card-shuffling mixing is governed by a coupon-collector mechanism. Diaconis, Fill, and Pitman determined the cutoff profile for the classical top-to-random shuffle on $S_n$ \cite{DFP}.  Nakano, Sadahiro, and Sakurai later studied top-to-random shuffles on colored permutations, or equivalently on the wreath product $C_p\wr S_n$ \cite{NSS}.  In \cite{FengProfiles}, the author studied the $p$-colored top-$m$-to-random shuffle with $m$ fixed and obtained exact formulas and cutoff profiles for total variation, separation, $L^q(U)$, $\chi^2$, relative entropy, and $L^\infty(U)$.

Motivated by the group-drawing coupon-collector regimes of Berend and Sher~\cite{BerendSher}, we let the block size \(m=m_n\) vary with \(n\).  Since the exact algebraic reduction of \cite{FengProfiles} holds for every \(1\le m\le n\), the remaining asymptotic problem is to understand the untouched-label statistic \(E_k^{(m)}\): after \(k\) independent uniform \(m\)-subset draws from \([n]\), how many labels are never drawn?

This is the classical collector's problem with group drawings. Stadje studied the model in \cite{Stadje}; related Poisson-approximation background includes the survey of Chatterjee, Diaconis, and Meckes \cite{CDM}, Schilling and Henze's Poisson limit theorems \cite{SchillingHenze}, Betken and Th\"ale's Stein-method bounds \cite{BetkenThale}, and recent asymptotic work of Berend and Sher \cite{BerendSher}. For our application we need a short, flexible Poisson criterion that remains valid when $m=m_n$ ranges from $o(n)$ to a positive fraction of $n$, and even when $m_n$ is very close to $n$.

Let $m_n\in\{1,\ldots,n-1\}$, let $b_n=n-m_n$, and let $E_n$ denote the number of labels never selected after $k_n$ independent uniform $m_n$-subset drawings.  Equivalently,
\[
        E_n=\left|B_1\cap\cdots\cap B_{k_n}\right|,
\]
where $B_1,\ldots,B_{k_n}$ are independent uniform $b_n$-subsets of $[n]$. When \(k_n=0\), we use the convention \(E_n=n\). The natural clock is the expected number of untouched labels,
\[
        \lambda_n
        :=
        \E[E_n]
        =
        n\left(\frac{b_n}{n}\right)^{k_n}.
\]
The non-microscopic-complement condition \(b_n\to\infty\) is the clean critical-regime hypothesis. Equivalently, when \(\lambda_n\to\lambda\in(0,\infty)\), it is the same as
\[
        \gamma_n:=\frac{k_nm_n}{n(n-m_n)}\to0.
\]
The quantity \(\gamma_n\) measures the cumulative dependence caused by sampling without replacement inside each round. We use the following Poisson limit.

\begin{theorem}[Variable-block group-drawing Poisson limit]\label{thm:coupon-poisson-intro}
Assume
\[
        \lambda_n=n\left(\frac{n-m_n}{n}\right)^{k_n}\longrightarrow \lambda\in(0,\infty),
        \qquad
        b_n=n-m_n\longrightarrow\infty.
\]
Then
\[
        E_n\Rightarrow \Poisson(\lambda).
\]
Equivalently, for every fixed integer $u\ge0$,
\[
        \Pbb(E_n=u)\longrightarrow e^{-\lambda}\frac{\lambda^u}{u!}.
\]
\end{theorem}

The proof is a factorial-moment calculation.  For fixed $r$,
\[
        \E\fall{E_n}{r}
        =\fall{n}{r}\left(\frac{\binom{n-r}{m_n}}{\binom{n}{m_n}}\right)^{k_n}
        =\fall{n}{r}\left(\frac{\fall{b_n}{r}}{\fall{n}{r}}\right)^{k_n},
\]
and the logarithm of the ratio of this expression to \(\lambda_n^r\) is 
\[O_r\!\left(\frac1n+\frac{k_nm_n}{n(n-m_n)}\right)=o(1). \]

We then combine Theorem \ref{thm:coupon-poisson-intro} with the exact nested-set formula from \cite{FengProfiles}.  Let $Q_{n,p}^{(m)}$ be the one-step law of $p$-colored top-$m$-to-random on $G_{n,p}=C_p\wr S_n$, and let $U_{n,p}$ be uniform.  For a continuous function $\Phi:[0,\infty)\to\R$ of polynomial growth, define
\[
        \mathcal I_{\Phi}(M,U)=\sum_x U(x)\Phi\!\left(\frac{M(x)}{U(x)}\right).
\]
For $\lambda>0$ and fixed $p\ge1$, put
\[
        a_{\ell,p}=\frac{1}{\ell!p^\ell}-\frac{1}{(\ell+1)!p^{\ell+1}},
        \qquad
        s_{\ell,p}(\lambda)=e^{-\lambda}\sum_{u=0}^{\ell}(p\lambda)^u.
\]
The limiting functional is
\[
        \mathcal I_{\Phi,p}(\lambda)
        :=\sum_{\ell=0}^{\infty}a_{\ell,p}\Phi(s_{\ell,p}(\lambda)).
\]
The shuffle consequence is this.

\begin{theorem}[Growing-block colored shuffle profiles]\label{thm:shuffle-intro}
Fix $p\ge1$.  Let $m_n,k_n$ satisfy the assumptions of Theorem \ref{thm:coupon-poisson-intro}, and set
\[
        M_n=(Q_{n,p}^{(m_n)})^{*k_n},\qquad U_n=U_{n,p}.
\]
Then, for every continuous $\Phi:[0,\infty)\to\R$ of polynomial growth,
\[
        \mathcal I_{\Phi}(M_n,U_n)\to \mathcal I_{\Phi,p}(\lambda).
\]
In particular, with $c=-\log\lambda$,
\[
        \normTV{M_n-U_n}\to f_p(c),
        \qquad
        f_p(c):=\sum_{\ell=0}^{\infty}a_{\ell,p}\bigl(s_{\ell,p}(e^{-c})-1\bigr)_+.
\]
Moreover,
\[
        \sepd(M_n,U_n)\to
        \begin{cases}
        1-e^{-\lambda}(1+\lambda),&p=1,\\
        1-e^{-\lambda},&p\ge2.
        \end{cases}
\]
\end{theorem}

Thus the fixed-\(m\) profiles persist whenever the untouched-label count has the required Poisson limit. The hypotheses of Theorem~\ref{thm:coupon-poisson-intro} are then verified in three concrete regimes. First, in the small-block regime \(m_n=o(n)\), the usual cutoff coordinate
\[
        k_n(c)=\left\lfloor\frac{\log n+c}{\delta_n}\right\rfloor,
        \qquad
        \delta_n=-\log\left(1-\frac{m_n}{n}\right),
\]
gives the profile \(f_p(c)\).  Second, in the proportional-block regime \(m_n/n\to\alpha\in(0,1)\), the cutoff window has order one; if \(\log n-k_n\delta_n\to a\), then the profile is \(f_p(-a)\). Third, for an integer \(r\ge2\) and \(a>0\), in the near-full regime
\[
        n-m_n\sim a n^{1-1/r},
\]
the nontrivial profile occurs after \(r\) shuffles and is \(f_p(-r\log a)\). This three-regime organization parallels the package-size regimes considered by Berend and Sher~\cite{BerendSher}.

Equivalently, the untouched set after \(k_n\) steps is the random intersection
\[
        B_1\cap\cdots\cap B_{k_n}
\]
of independent uniform \(b_n\)-subsets of \([n]\).  Thus the probabilistic content of the paper is a Poisson limit for the size of this random intersection, while the shuffle-specific content enters through the exact nested-set likelihood-ratio reduction.

The paper is organized as follows. Section \ref{sec:reduction} recalls the exact nested-set reduction for colored top-$m$-to-random shuffles. Section \ref{sec:coupon} proves the variable-block Poisson limit. Section \ref{sec:profiles} proves the shuffle profile theorem. Section \ref{sec:regimes} gives the main regimes: $m_n=o(n)$, proportional blocks, and near-full blocks.

\section{The exact shuffle reduction}\label{sec:reduction}

Fix integers $n,p\ge1$.  Let
\[
        G_{n,p}=C_p\wr S_n\cong C_p^n\rtimes S_n
\]
be the group of $p$-colored permutations. We write $U_{n,p}$ for the uniform probability measure on $G_{n,p}$.

\begin{definition}[Colored top-$m$-to-random]\label{def:shuffle}
For $1\le m\le n$, one step of $p$-colored top-$m$-to-random is the following shuffle. Remove the top $m$ cards, recolor those $m$ cards by independent uniform elements of $C_p$, apply a uniform random permutation to the removed block, choose an $m$-element set of final positions uniformly among all $\binom nm$ possibilities, and insert the removed cards into those positions. The remaining $n-m$ cards preserve their relative order and color.  Let $Q_{n,p}^{(m)}$ denote the one-step law on $G_{n,p}$.
\end{definition}

For $0\le u\le n$, let $A_u\subseteq G_{n,p}$ be the set of colored permutations in which the largest $u$ labels occur in increasing order and have color $0$. Thus
\[
        A_0=G_{n,p},\qquad A_0\supseteq A_1\supseteq\cdots\supseteq A_n.
\]
Define
\[
        L_p(x)=\max\{u\in\{0,\ldots,n\}:x\in A_u\}.
\]
Under the uniform law $U_{n,p}$,
\begin{equation}\label{eq:Lp-dist}
        U_{n,p}\{L_p=\ell\}
        =\frac{1}{\ell!p^\ell}-\frac{1}{(\ell+1)!p^{\ell+1}}\qquad(0\le \ell<n),
\end{equation}
while
\[
        U_{n,p}\{L_p=n\}=\frac{1}{n!p^n}.
\]
Indeed, $U_{n,p}(A_u)=1/(u!p^u)$.

Now fix $n,m,k$. Let $S_1,\ldots,S_k$ be independent uniform $m$-subsets of $[n]$, and let
\[
        T_k=S_1\cup\cdots\cup S_k,
        \qquad
        E_k^{(m)}=n-|T_k|.
\]
Equivalently, if $B_i=[n]\setminus S_i$, then $B_1,
\ldots,B_k$ are independent uniform $(n-m)$-subsets and
\[
        E_k^{(m)}=|B_1\cap\cdots\cap B_k|.
\]
Let
\[
        P_k^{(m)}(u)=\Pbb(E_k^{(m)}=u),\qquad 0\le u\le n.
\]

We use the following exact reduction from \cite{FengProfiles}.

\begin{proposition}[Exact nested-set reduction {\cite{FengProfiles}}]\label{thm:exact-reduction}
For every $n,p\ge1$, $1\le m\le n$, $k\ge0$, and $x\in G_{n,p}$,
\begin{equation}\label{eq:LR-exact}
        \frac{(Q_{n,p}^{(m)})^{*k}(x)}{U_{n,p}(x)}
        =\sum_{u=0}^{L_p(x)}P_k^{(m)}(u)u!p^u.
\end{equation}
Consequently, for every continuous $\Phi:[0,\infty)\to\R$,
\begin{equation}\label{eq:functional-exact}
        \sum_{x\in G_{n,p}}U_{n,p}(x)\Phi\!\left(\frac{(Q_{n,p}^{(m)})^{*k}(x)}{U_{n,p}(x)}\right)
        =\sum_{\ell=0}^{n}U_{n,p}\{L_p=\ell\}\Phi\!\left(\sum_{u=0}^{\ell}P_k^{(m)}(u)u!p^u\right).
\end{equation}
Moreover,
\begin{equation}\label{eq:sep-exact}
        \sepd\!\left((Q_{n,p}^{(m)})^{*k},U_{n,p}\right)
        =
        \begin{cases}
        1-P_k^{(m)}(0)-P_k^{(m)}(1),&p=1,\\
        1-P_k^{(m)}(0),&p\ge2.
        \end{cases}
\end{equation}
\end{proposition}

\begin{proof}
The likelihood-ratio identity is Corollary~4.5 of \cite{FengProfiles}; the functional identity is Corollary~5.20; and the separation formula is Theorem~5.7, applied with the present \(m\) and \(k\).
\end{proof}

Crucially, \eqref{eq:LR-exact} holds for every \(m\), including \(m=m_n\).

\section{A Poisson theorem for variable-size group drawings}\label{sec:coupon}

In this section \(m=m_n\) and \(k=k_n\) may vary with \(n\).  We always assume
\[
        1\le m_n\le n-1,
        \qquad
        b_n=n-m_n,
        \qquad
        q_n=\frac{b_n}{n},
        \qquad
        \lambda_n=nq_n^{k_n}.
\]
Let \(E_n=E_{k_n}^{(m_n)}\) be the number of untouched labels after \(k_n\) uniform \(m_n\)-subset drawings.

We use the falling factorial notation
\[
        \fall{x}{r}=x(x-1)\cdots(x-r+1),\qquad \fall{x}{0}=1.
\]

\begin{lemma}[Factorial moments]\label{lem:factorial-moment}
Assume \(k_n\ge1\). For every integer \(r\) with \(0\le r\le n\),
\begin{equation}\label{eq:factorial-moment}
        \E\fall{E_n}{r}
        =\fall{n}{r}\left(\frac{\binom{n-r}{m_n}}{\binom{n}{m_n}}\right)^{k_n}
        =\fall{n}{r}\left(\frac{\fall{b_n}{r}}{\fall{n}{r}}\right)^{k_n},
\end{equation}
where the expression is interpreted as $0$ if $r>b_n$.
\end{lemma}

\begin{proof}
This is Lemma~4.7 of \cite{FengProfiles} applied with \(m=m_n\), \(k=k_n\), and \(u=r\).  The second equality follows from
\[
        \frac{\binom{n-r}{m_n}}{\binom{n}{m_n}}
        =
        \frac{\fall{b_n}{r}}{\fall{n}{r}},
        \qquad b_n=n-m_n. \qedhere
\]
\end{proof}

\begin{lemma}[Critical-regime form of the dependence condition]\label{lem:critical-equivalence}
Assume
\[
        \lambda_n\to\lambda\in(0,\infty).
\]
Then
\[
        \frac{k_nm_n}{nb_n}\to0
        \qquad\Longleftrightarrow\qquad
        b_n\to\infty.
\]
\end{lemma}

\begin{proof}
First assume
\[
        \frac{k_nm_n}{nb_n}\to0.
\]
Since \(k_n=0\) would give \(\lambda_n=n\to\infty\), the convergence \(\lambda_n\to\lambda<\infty\) implies \(k_n\ge1\) for all sufficiently large \(n\). Hence
\[
        \frac{k_nm_n}{nb_n}
        \ge
        \frac{m_n}{nb_n}
        =
        \frac1{b_n}-\frac1n.
\]
Thus
\[
        0\le \frac1{b_n}\le \frac{k_nm_n}{nb_n}+\frac1n\to0,
\]
so \(b_n\to\infty\).

Conversely, assume \(b_n\to\infty\). Choose \(c>0\) such that \(\lambda_n\ge c\) for all sufficiently large \(n\).  Since \(k_n=0\) would give \(\lambda_n=n\to\infty\), and \(k_n=1\) would give \(\lambda_n=b_n\to\infty\), we have \(k_n\ge2\) for all sufficiently large \(n\). From
\[
        nq_n^{k_n}=\lambda_n\ge c
\]
we get, for all sufficiently large \(n\),
\[
        q_n\ge \left(\frac cn\right)^{1/k_n}
        \ge
        \left(\frac cn\right)^{1/2},
\]
and therefore
\[
        b_n=nq_n\ge \sqrt{cn}.
\]
Now put
\[
        \delta_n=-\log q_n.
\]
Since \(\lambda_n=nq_n^{k_n}\),
\[
        k_n\delta_n
        =
        \log n-\log\lambda_n
        =
        O(\log n).
\]
Using \(1-q_n\le -\log q_n=\delta_n\), we obtain
\[
        \frac{k_nm_n}{nb_n}
        =
        \frac{k_n(1-q_n)}{b_n}
        \le
        \frac{k_n\delta_n}{b_n}
        =
        O\!\left(\frac{\log n}{\sqrt n}\right)
        \to0.
\]
This proves the equivalence.
\end{proof}

\begin{theorem}[Variable-block Poisson limit]\label{thm:coupon-poisson}
Assume
\begin{equation}\label{eq:main-assumptions}
        \lambda_n=nq_n^{k_n}\to\lambda\in(0,\infty),
        \qquad
        b_n=n-m_n\to\infty.
\end{equation}
Then
\[
        E_n\Rightarrow \Poisson(\lambda).
\]
Equivalently, for each fixed integer $u\ge0$,
\[
        \Pbb(E_n=u)\to e^{-\lambda}\frac{\lambda^u}{u!}.
\]
\end{theorem}

\begin{proof}
By Lemma~\ref{lem:critical-equivalence},
\[
        \frac{k_nm_n}{nb_n}\to0.
\]
We prove convergence of factorial moments. Since \(\lambda_n\to\lambda<\infty\), we have \(k_n\ge1\) for all sufficiently large \(n\).

Fix $r\ge1$. By Lemma~\ref{lem:factorial-moment}, for all sufficiently large $n$,
\[
        \E\fall{E_n}{r}
        =
        \fall{n}{r}
        \left(\frac{\fall{b_n}{r}}{\fall{n}{r}}\right)^{k_n}.
\]
Compare this to
\[
        \lambda_n^r
        =
        n^r\left(\frac{b_n}{n}\right)^{k_nr}.
\]
Taking logarithms,
\begin{align*}
        \log\frac{\E\fall{E_n}{r}}{\lambda_n^r}
        &=
        \log\frac{\fall{n}{r}}{n^r}
        +
        k_n\sum_{j=0}^{r-1}
        \left[
        \log\left(1-\frac{j}{b_n}\right)
        -
        \log\left(1-\frac{j}{n}\right)
        \right].
\end{align*}
Since \(r\) is fixed,
\[
        \log\frac{\fall{n}{r}}{n^r}
        =
        O_r\!\left(\frac1n\right).
\]
Moreover, since \(b_n\to\infty\), the mean value theorem applied to \(x\mapsto\log(1-jx)\) gives, uniformly for \(0\le j\le r-1\),
\[
        \log\left(1-\frac{j}{b_n}\right)
        -
        \log\left(1-\frac{j}{n}\right)
        =
        O_r\!\left(\frac1{b_n}-\frac1n\right)
        =
        O_r\!\left(\frac{m_n}{nb_n}\right).
\]
Therefore
\[
        \log\frac{\E\fall{E_n}{r}}{\lambda_n^r}
        =
        O_r\!\left(\frac1n+\frac{k_nm_n}{nb_n}\right)
        =
        o(1).
\]
Thus
\[
        \E\fall{E_n}{r}\to\lambda^r
        \qquad(r\ge0).
\]
Ordinary moments are finite linear combinations of factorial moments, so all ordinary moments converge to the moments of \(\Poisson(\lambda)\). Since the Poisson law is moment-determinate,
\[
        E_n\Rightarrow \Poisson(\lambda).
\]

Finally, because each \(E_n\) is integer-valued, for every fixed integer \(u\ge0\),
\[
        \Pbb(E_n=u)
        =
        \Pbb\!\left(u-\frac12<E_n<u+\frac12\right).
\]
The boundary of this interval has zero \(\Poisson(\lambda)\)-mass, so weak convergence gives
\[
        \Pbb(E_n=u)\to e^{-\lambda}\frac{\lambda^u}{u!}. \qedhere
\]
\end{proof}

\begin{remark}[A non-Poisson boundary case]\label{rem:condition-needed}
The condition \(b_n\to\infty\) in Theorem~\ref{thm:coupon-poisson} cannot be removed. If \(m_n=n-b\) with fixed \(b\ge1\) and \(k_n=1\), then \(\lambda_n=b\), but \(E_n=b\) deterministically, not \(\Poisson(b)\).  Here \(b_n=b\) does not tend to infinity, and equivalently
\[
        \frac{k_nm_n}{n(n-m_n)}
        =
        \frac{n-b}{nb}
        \to
        \frac1b.
\]
\end{remark}

We also need the following uniform bound for the likelihood-ratio functional limit.

\begin{lemma}[A uniform exponential envelope]\label{lem:envelope}
Fix \(p\ge1\), and assume $\lambda_n=nq_n^{k_n}$ is bounded. Then, for all sufficiently large $n$ and every integer $u\ge0$,
\begin{equation}\label{eq:envelope}
        \Pbb(E_n=u)u!p^u\le (p\lambda_n)^u.
\end{equation}
Consequently, if $\lambda_n\to\lambda$, there is a constant $B\ge1$ such that, for all sufficiently large \(n\) and all integers \(u\ge0\),
\[
        \Pbb(E_n=u)u!p^u\le B^u.
\]
\end{lemma}

\begin{proof}
Since $\lambda_n$ is bounded, we have $k_n\ge1$ for all sufficiently large $n$.  We work only with such $n$.
For $u=0$ this is immediate.  If $u>b_n$, then $E_n\le b_n$ and the left-hand side is $0$. We may therefore assume $1\le u\le b_n$.  Then
\[
        \Pbb(E_n=u)u!\le \E\fall{E_n}{u}.
\]
By Lemma \ref{lem:factorial-moment}, and using
\[
        \frac{b_n-j}{n-j}\le \frac{b_n}{n}=q_n\qquad(0\le j<u),
\]
we get
\[
        \E\fall{E_n}{u}
        =\fall{n}{u}\prod_{j=0}^{u-1}\left(\frac{b_n-j}{n-j}\right)^{k_n}
        \le n^u q_n^{k_nu}=\lambda_n^u.
\]
Multiplying by $p^u$ proves \eqref{eq:envelope}.
\end{proof}

\section{Growing-block shuffle profiles}\label{sec:profiles}

Fix $p\ge1$. For $\lambda>0$, define
\begin{equation}\label{eq:aell}
        a_{\ell,p}=\frac{1}{\ell!p^\ell}-\frac{1}{(\ell+1)!p^{\ell+1}},
        \qquad \ell\ge0,
\end{equation}
with the convention that $0!=1$. Also define
\begin{equation}\label{eq:sell}
        s_{\ell,p}(\lambda)=e^{-\lambda}\sum_{u=0}^{\ell}(p\lambda)^u,
        \qquad \ell\ge0.
\end{equation}
The numbers $(a_{\ell,p})_{\ell\ge0}$ form a probability mass function, since the series telescopes.

For a probability measure $M$ on $G_{n,p}$ and a function $\Phi:[0,\infty)\to\R$, write
\[
        \mathcal I_{\Phi}(M,U_{n,p})
        =\sum_{x\in G_{n,p}}U_{n,p}(x)\Phi\!\left(\frac{M(x)}{U_{n,p}(x)}\right).
\]

\begin{theorem}[Likelihood-ratio functional limit]\label{thm:functional-limit}
Fix $p\ge1$.  Let $m_n,k_n$ satisfy \eqref{eq:main-assumptions}, and set
\[
        M_n=(Q_{n,p}^{(m_n)})^{*k_n},\qquad U_n=U_{n,p}.
\]
If $\Phi:[0,\infty)\to\R$ is continuous and has polynomial growth, meaning
\[
        |\Phi(t)|\le C(1+t^d)\qquad(t\ge0)
\]
for some \(C>0\) and \(d\ge0\), then
\begin{equation}\label{eq:functional-limit}
        \mathcal I_{\Phi}(M_n,U_n)\to
        \sum_{\ell=0}^{\infty}a_{\ell,p}\Phi(s_{\ell,p}(\lambda)).
\end{equation}
\end{theorem}

\begin{proof}
Set \(\mu_n(n-u):=\Pbb(E_n=u)\) for \(0\le u\le n\).  By the exact top-\(m\) occupancy mixture formula in \cite[Lemma~4.3]{FengProfiles}, \(M_n\) is the nested-set mixture associated with \(\mu_n\). By Theorem~\ref{thm:coupon-poisson},
\[
        \mu_n(n-u)\to e^{-\lambda}\frac{\lambda^u}{u!}
        \qquad(u\ge0\text{ fixed integer}),
\]
and by Lemma~\ref{lem:envelope} there is \(B\ge1\) such that
\[
        \mu_n(n-u)u!p^u\le B^u
\]
for all sufficiently large \(n\) and all integers \(0\le u\le n\). Therefore the hypotheses of \cite[Theorem~5.27]{FengProfiles} hold. Applying that theorem gives
\[
        \mathcal I_{\Phi}(M_n,U_n)\to
        \sum_{\ell=0}^{\infty}a_{\ell,p}\Phi(s_{\ell,p}(\lambda)). \qedhere
\]
\end{proof}

Several standard distances are obtained by choosing $\Phi$.

\begin{corollary}[Total variation and other integrated metrics]\label{cor:metrics}
Under the hypotheses of Theorem \ref{thm:functional-limit},
\begin{align*}
        \normTV{M_n-U_n}
        &\to
        \sum_{\ell=0}^{\infty}a_{\ell,p}\bigl(s_{\ell,p}(\lambda)-1\bigr)_+,\\
        \sum_{x\in G_{n,p}}U_n(x)\left|\frac{M_n(x)}{U_n(x)}-1\right|^q
        &\to
        \sum_{\ell=0}^{\infty}a_{\ell,p}|s_{\ell,p}(\lambda)-1|^q
        \qquad(1\le q<\infty),\\
        \chi^2(M_n,U_n)
        &\to
        \sum_{\ell=0}^{\infty}a_{\ell,p}(s_{\ell,p}(\lambda)-1)^2,\\
        D(M_n\|U_n)
        &\to
        \sum_{\ell=0}^{\infty}a_{\ell,p}s_{\ell,p}(\lambda)\log s_{\ell,p}(\lambda).
\end{align*}
Here $D$ denotes relative entropy, with the convention $0\log0=0$.
\end{corollary}

\begin{proof}
Take respectively
\[
        \Phi(t)=(t-1)_+,
        \qquad
        \Phi(t)=|t-1|^q,
        \qquad
        \Phi(t)=(t-1)^2,
        \qquad
        \Phi(t)=t\log t.
\]
For total variation we use
\[
        \normTV{M-U}=\sum_x U(x)\left(\frac{M(x)}{U(x)}-1\right)_+,
\]
which follows from $\sum_x U(x)M(x)/U(x)=1$.
\end{proof}

It is convenient to rewrite the total variation limit in cutoff-window notation.  For $c\in\R$, put $\lambda=e^{-c}$ and define
\begin{equation}\label{eq:fp}
        f_p(c)=\sum_{\ell=0}^{\infty}a_{\ell,p}\bigl(s_{\ell,p}(e^{-c})-1\bigr)_+.
\end{equation}
This is exactly the fixed-$m$ total variation profile from \cite{FengProfiles}.

\begin{corollary}[Separation profile]\label{cor:sep}
Under the hypotheses of Theorem \ref{thm:functional-limit},
\[
        \sepd(M_n,U_n)\to
        \begin{cases}
        1-e^{-\lambda}(1+\lambda),&p=1,\\
        1-e^{-\lambda},&p\ge2.
        \end{cases}
\]
\end{corollary}

\begin{proof}
The exact separation formula \eqref{eq:sep-exact} gives
\[
        \sepd(M_n,U_n)=
        \begin{cases}
        1-\Pbb(E_n=0)-\Pbb(E_n=1),&p=1,\\
        1-\Pbb(E_n=0),&p\ge2.
        \end{cases}
\]
Theorem \ref{thm:coupon-poisson} gives
\[
        \Pbb(E_n=0)\to e^{-\lambda},
        \qquad
        \Pbb(E_n=1)\to \lambda e^{-\lambda}.\qedhere
\]
\end{proof}

\section{Concrete regimes}\label{sec:regimes}

The general theorem is stated in terms of $\lambda_n=n((n-m_n)/n)^{k_n}$ because that is the exact coupon clock.  This section translates it into the common growth regimes for $m_n$.

\subsection[Small blocks: mn=o(n)]{Small blocks: \texorpdfstring{$m_n=o(n)$}{mn=o(n)}}

Let
\[
        \delta_n=-\log\left(1-\frac{m_n}{n}\right).
\]
If $m_n=o(n)$, then $\delta_n\sim m_n/n$.  For fixed $c\in\R$, set
\begin{equation}\label{eq:small-block-time}
        k_n(c)=\left\lfloor\frac{\log n+c}{\delta_n}\right\rfloor.
\end{equation}

\begin{corollary}[Small-block cutoff profile]\label{cor:small-block}
Assume $m_n=o(n)$ and $1\le m_n\le n-1$.  For each fixed \(c\in\R\), let \(k_n(c)\) be as in \eqref{eq:small-block-time}.
\[ \lambda_n = n\left(1-\frac{m_n}{n}\right)^{k_n(c)} \to e^{-c}, \qquad b_n=n-m_n\to\infty. \]
Consequently, for every fixed $p\ge1$,
\[
        \normTV{(Q_{n,p}^{(m_n)})^{*k_n(c)}-U_{n,p}}\to f_p(c),
\]
and
\[
        \sepd\!\left((Q_{n,p}^{(m_n)})^{*k_n(c)},U_{n,p}\right)\to
        \begin{cases}
        1-e^{-e^{-c}}(1+e^{-c}),&p=1,\\
        1-e^{-e^{-c}},&p\ge2.
        \end{cases}
\]
In particular the total variation cutoff location is
\[
        \frac{\log n}{\delta_n}\sim \frac{n}{m_n}\log n
\]
with window
\[
        \frac1{\delta_n}\sim\frac{n}{m_n}.
\]
\end{corollary}

\begin{proof}
Since $k_n(c)=((\log n+c)/\delta_n)+O(1)$,
\[
        n\left(1-\frac{m_n}{n}\right)^{k_n(c)}
        =\exp\{\log n-k_n(c)\delta_n\}\to e^{-c},
\]
because $\delta_n\to0$.  Moreover, since \(m_n=o(n)\),
\[
        n-m_n\sim n\to\infty.
\]
The profile conclusions follow from Corollaries \ref{cor:metrics} and \ref{cor:sep}.
Since $f_p(c)\to1$ as $c\to-\infty$ and $f_p(c)\to0$ as $c\to+\infty$ by \cite{FengProfiles}, this gives the stated cutoff location and window.
\end{proof}

\subsection{Proportional blocks and lattice phases}

Suppose
\[
        \frac{m_n}{n}\to\alpha\in(0,1),
        \qquad
        \delta_n=-\log\left(1-\frac{m_n}{n}\right)\to\delta=-\log(1-
\alpha)>0.
\]
The cutoff window is now of order one. Therefore the integer part in the time parameter no longer disappears automatically.

\begin{corollary}[Proportional blocks]\label{cor:proportional}
Assume $m_n/n\to\alpha\in(0,1)$, and let $k_n$ be any integer sequence such that
\begin{equation}\label{eq:lattice-a}
        \log n-k_n\delta_n\to a\in\R.
\end{equation}
Then, with $\lambda=e^a$, for every fixed \(p\ge1\),
\[
        \normTV{(Q_{n,p}^{(m_n)})^{*k_n}-U_{n,p}}\to f_p(-a),
\]
and
\[
        \sepd\!\left((Q_{n,p}^{(m_n)})^{*k_n},U_{n,p}\right)\to
        \begin{cases}
        1-e^{-e^a}(1+e^a),&p=1,\\
        1-e^{-e^a},&p\ge2.
        \end{cases}
\]
\end{corollary}

\begin{proof}
Condition \eqref{eq:lattice-a} says precisely that
\[
        \lambda_n=n\exp(-k_n\delta_n)\to e^a.
\]
Also
\[
        n-m_n\sim (1-\alpha)n\to\infty.
\]
Apply Corollaries~\ref{cor:metrics} and~\ref{cor:sep}.
\end{proof}

\begin{remark}[The lattice effect]
The proportional-block assumption \(m_n/n\to\alpha\in(0,1)\) does not by itself determine a limiting profile; the additional condition
\[
        \log n-k_n\delta_n\to a
\]
is precisely the full-sequence critical-window condition, since it is equivalent to
\[
        \lambda_n
        =
        \exp\{\log n-k_n\delta_n\}
        \to e^a.
\]
For fixed \(c\in\R\), we might try
\[
        k_n(c)=\left\lfloor\frac{\log n+c}{\delta_n}\right\rfloor.
\]
Then
\[
        \log n-k_n(c)\delta_n=-c+\theta_n(c)\delta_n,
        \qquad
        \theta_n(c)=\left\{\frac{\log n+c}{\delta_n}\right\}\in[0,1),
\]
where \(\{\cdot\}\) denotes fractional part. Since \(\delta_n\to\delta>0\), the term \(\theta_n(c)\delta_n\) need not vanish. Thus if
\(\theta_n(c)\) does not converge, the total variation distance need not have a full-sequence limit. Along any subsequence for which \(\theta_n(c)\to\theta\), Corollary~\ref{cor:proportional} gives the subsequential limit \(f_p(c-\theta\delta)\).
\end{remark}

\subsection{Near-full blocks}

The theorem also covers near-full blocks, where each shuffle moves almost the entire deck. In the corollary below, the assumption
\[
        n-m_n\sim a n^{1-1/r}
\]
with \(a>0\) and integer \(r\ge2\) is indeed near-full: since \(n^{1-1/r}=o(n)\), it implies \(n-m_n=o(n)\), and hence \(m_n/n\to1\).

\begin{corollary}[Polynomial-size complements]\label{cor:near-full}
Fix an integer $r\ge2$ and a number $a>0$. Suppose
\[
       n-m_n\sim a n^{1-1/r},
        \qquad
        k_n=r.
\]
Then
\[
        E_r^{(m_n)}\Rightarrow \Poisson(a^r),
\]
and for every fixed $p\ge1$,
\[
        \normTV{(Q_{n,p}^{(m_n)})^{*r}-U_{n,p}}\to f_p(-r\log a),
\]
while
\[
        \sepd\!\left((Q_{n,p}^{(m_n)})^{*r},U_{n,p}\right)\to
        \begin{cases}
        1-e^{-a^r}(1+a^r),&p=1,\\
        1-e^{-a^r},&p\ge2.
        \end{cases}
\]
\end{corollary}

\begin{proof}
The assumption gives
\[
        b_n= n-m_n\sim a n^{1-1/r}\to\infty.
\]
Moreover, since \(k_n=r\),
\[
        \lambda_n
        =
        n\left(\frac{n-m_n}{n}\right)^r
        =
        n\left(\frac{b_n}{n}\right)^r
        =
        \left(\frac{b_n}{n^{1-1/r}}\right)^r
        \to a^r.
\]
Apply Theorem~\ref{thm:coupon-poisson}, Corollary~\ref{cor:metrics}, and Corollary~\ref{cor:sep}.
\end{proof}

\begin{example}
Let \(a>0\), \(m_n=n-a\sqrt n+o(\sqrt n)\), and \(k_n=2\).  Then
\[
        n-m_n\sim a\sqrt n=a n^{1-1/2},
\]
so this is the near-full regime with \(r=2\).  Moreover,
\[
        \lambda_n
        =
        n\left(\frac{n-m_n}{n}\right)^2
        \sim
        n\left(\frac{a\sqrt n}{n}\right)^2
        =
        a^2.
\]
Hence
\[
        E_2^{(m_n)}\Rightarrow \Poisson(a^2).
\]
Since \(a^2=e^{-c}\) gives \(c=-2\log a\), the limiting total variation
distance is
\[
        \normTV{(Q_{n,p}^{(m_n)})^{*2}-U_{n,p}}
        \to
        f_p(-2\log a).
\]
Thus two near-full shuffles can still leave a nontrivial amount of memory.
\end{example}

\section{Closing remarks}

The profiles are governed by the coupon clock
\[
        n\left(1-\frac{m_n}{n}\right)^{k_n}.
\]
At the critical scale where this quantity tends to a finite nonzero limit, the required non-microscopic-complement condition is
\[
        n-m_n\to\infty,
        \qquad\text{equivalently}\qquad
        \frac{k_nm_n}{n(n-m_n)}\to0.
\]
Section~\ref{sec:regimes} gives three concrete instances of this criterion.

It would be natural to push further in two directions. First, we can ask for quantitative error terms in total variation, using Stein bounds for group drawings in the spirit of \cite{BetkenThale}. Second, we can ask what replaces the Poisson profile when $\gamma_n$ does not vanish. Remark~\ref{rem:condition-needed} gives a non-Poisson boundary case when $n-m_n$ is fixed and $k_n=1$.

\section*{Acknowledgments}
The author thanks Jason Fulman for helpful feedback and for suggesting references.

\end{document}